\documentclass{amsart}

\usepackage{amssymb}
\usepackage{amsmath}
\usepackage{amsthm}
\usepackage{graphicx,color}

\allowdisplaybreaks
\numberwithin{equation}{section}

\theoremstyle{plain}
\newtheorem{thm}{Theorem}[section]

\newtheorem{lem}[thm]{Lemma}

\theoremstyle{definition}

\newtheorem{rem}[thm]{Remark}

\newcommand{\R}{\mathbb{R}}

\newcommand{\Z}{\mathbb{Z}}

\newcommand{\calF}{\mathcal{F}}

\newcommand{\calM}{\mathcal{M}}
\newcommand{\calS}{\mathcal{S}}

\begin{document}

\title[A counterexample to multilinear weighted estimates]
{A counterexample to weighted estimates \\
for multilinear Fourier multipliers with \\
Sobolev regularity}
\author{Mai Fujita \and Naohito Tomita}
%\date{}

\address{Mai Fujita \\
Department of Mathematics \\
Graduate School of Science \\
Osaka University \\
Toyonaka, Osaka 560-0043, Japan}
\email{m-fujita@cr.math.sci.osaka-u.ac.jp}

\address{Naohito Tomita \\
Department of Mathematics \\
Graduate School of Science \\
Osaka University \\
Toyonaka, Osaka 560-0043, Japan}
\email{tomita@math.sci.osaka-u.ac.jp}



%%%==================================================================
\begin{abstract}
The problem whether weighted estimates
for multilinear Fourier multipliers with Sobolev regularity
hold under weak condition on weights is considered.
\end{abstract}

%%%==================================================================
\maketitle

%%%==================================================================
%%%==================================================================
\section{Introduction}\label{section1}
%%%==================================================================
In this paper,
we consider weighted norm inequalities
for multilinear Fourier multipliers with Sobolev regularity.
Before discussing them, we briefly recall some basic facts
on weights in the multilinear theory.

In the linear case,
it is well known that the Hardy-Littlewood maximal operator $M$
is bounded on $L^p(w)$ if and only if the weight $w$ belongs to the
Muckenhoupt class $A_p$,
where $1<p<\infty$ and $M$ is defined by
\[
Mf(x)=\sup_{Q \ni x}\frac{1}{|Q|}\int_Q |f(y)|\, dy
\]
for $f \in L^1_{\mathrm{loc}}(\R^n)$
(see Section \ref{section2} for the definitions not given in this section).
Let $1<p_1,\dots,p_N<\infty$ and $1/p_1+\dots+1/p_N=1/p$.
For $\vec{w}=(w_1,\dots,w_N)$,
we set $\nu_{\vec{w}}=\prod_{k=1}^Nw_k^{p/p_k}$.
By H\"older's inequality, if
$\vec{w}=(w_1,\dots,w_N) \in
A_{p_1} \times \dots \times A_{p_N}$,
then
\begin{equation}\label{maximal-1}
\left\|\prod_{k=1}^N Mf_k \right\|_{L^p(\nu_{\vec{w}})}
\lesssim
\prod_{k=1}^N
\|f_k\|_{L^{p_k}(w_k)}.
\end{equation}
We define the multi(sub)linear maximal operator $\calM$ by
\[
\calM(\vec{f})(x)
=\sup_{Q \ni x}\prod_{k=1}^N
\frac{1}{|Q|}\int_Q |f_k(y_k)|\, dy_k
\]
for $\vec{f}=(f_1,\dots,f_N) \in L^1_{\mathrm{loc}}(\R^n)^N$,
and note that
\[
\calM(\vec{f})(x)
\le \prod_{k=1}^N Mf_k(x).
\]
Lerner, Ombrosi, P\'erez, Torres and Trujillo-Gonz\'alez \cite{LOPTT}
proved that $\calM$ is bounded from
$L^{p_1}(w_1) \times \dots \times L^{p_N}(w_N)$
to $L^p(\nu_{\vec{w}})$ if and only if
$\vec{w} \in A_{(p_1,\dots,p_N)}$.
It should be remarked that
there exists
$\vec{w} \in A_{(p_1,\dots,p_N)}$ such that
\eqref{maximal-1} does not hold (\cite[Remark 7.5]{LOPTT}).
This says that the inclusion
\begin{equation}\label{weight-inclusion}
A_{p_1} \times \dots \times A_{p_N}
\subset A_{(p_1,\dots,p_N)}
\end{equation}
is strict.

Let $m \in L^{\infty}(\R^{Nn})$.
The $N$-linear Fourier multiplier operator $T_m$
is defined by
\[
T_m(\vec{f})(x)
=\frac{1}{(2\pi)^{Nn}}
\int_{\R^{Nn}}
e^{ix\cdot(\xi_1+\dots+\xi_N)}
m(\xi)\widehat{f_1}(\xi_1) \dots \widehat{f_N}(\xi_N)\,
d\xi
\]
for $\vec{f}=(f_1,\dots, f_N) \in \calS(\R^n)^N$,
where $x \in \R^n$,
$\xi=(\xi_1,\dots,\xi_N) \in \R^n \times \dots \times \R^n$
and $d\xi=d\xi_1,\dots,d\xi_N$.
It is well known that in the unweighted case
the boundedness of $T_m$ holds if
\[
|\partial_{\xi}^{\alpha}m(\xi)|
\le C_{\alpha}|\xi|^{-|\alpha|}
\]
for sufficiently many multi-indices $\alpha$
(see, for example, \cite{CM, G-T, Kenig-Stein}
and also \cite{Grafakos-Miyachi-Tomita, Grafakos-Si, Miyachi-Tomita, Tomita}
for multipliers with Sobolev regularity).
We set
\begin{equation}\label{(multiplier)}
m_j(\xi)
=m(2^j\xi)\Psi(\xi),
\qquad j \in \Z,
\end{equation}
where $\Psi$ is a function in $\calS(\R^{Nn})$ satisfying
\[
\mathrm{supp}\, \Psi \subset
\{\xi \in \R^{Nn} : 1/2 \le |\xi| \le 2\},
\qquad
\sum_{k \in \Z}\Psi(\xi/2^k)=1,
\ \xi \in \R^{Nn} \setminus \{0\}.
\]
We use the notation 
$\|T_m\|_{L^{p_1}(w_1) \times \dots \times L^{p_N}(w_N)
\to L^p(\nu_{\vec{w}})}$
to denote the smallest constant $C$ satisfying 
\[
\|T_m(\vec{f})\|_{L^p(\nu_{\vec{w}})}
\leq C \prod_{k=1}^N
\|f_k\|_{L^{p_k}(w_k)}
\]
for all $\vec{f}=(f_1,\dots,f_N) \in \calS(\R^n)^N$.

Let $Nn/2 <s \le Nn$, 
$Nn/s<p_1,\dots,p_N<\infty$
and $1/p_1+\dots+1/p_N=1/p$.
It follows from \cite[Theorem 6.2]{Fujita-Tomita} that if
$\vec{w}=(w_1,\dots,w_N) \in
A_{p_1s/(Nn)} \times \dots \times A_{p_Ns/(Nn)}$,
then
\begin{equation}\label{weight-FT}
\|T_m\|_{L^{p_1}(w_1) \times \dots \times L^{p_N}(w_N) \to L^p(\nu_{\vec{w}})}
\lesssim
\sup_{j \in \Z}\|m_j\|_{W^{(s/N,\dots,s/N)}(\R^{Nn})},
\end{equation}
where the implicit constant is independent of $m$
(see Li, Xue and Yabuta \cite{Li-Xue-Yabuta} for the endpoint cases).
This result can also be obtained from another approach
of Hu and Lin \cite{Hu-Lin}.
Replacing $W^{(s/N,\dots,s/N)}$ by $W^s$,
Bui and Duong \cite{Bui-Duong},
Li and Sun \cite{Li-Sun} proved that if
$\vec{w}=(w_1,\dots,w_N) \in A_{(p_1s/(Nn),\dots,p_Ns/(Nn))}$,
then
\begin{equation}\label{weight-BD}
\|T_m\|_{L^{p_1}(w_1) \times \dots \times L^{p_N}(w_N) \to L^p(\nu_{\vec{w}})}
\lesssim
\sup_{j \in \Z}\|m_j\|_{W^s(\R^{Nn})}.
\end{equation}
By the embedding
\[
W^s(\R^{Nn}) \hookrightarrow
W^{(s/N,\dots,s/N)}(\R^{Nn}),
\]
we note that
the regularity condition in \eqref{weight-BD}
is stronger than that in \eqref{weight-FT}.
Of course,
it follows from \eqref{weight-inclusion} that
estimate \eqref{weight-BD} holds if
$\vec{w}=(w_1,\dots,w_N) \in
A_{p_1s/(Nn)} \times \dots \times A_{p_Ns/(Nn)}$.
See Table \ref{Table-1} for the three cases mentioned here.
If $N=1$ (namely, the linear case),
estimate \eqref{weight-FT} is the same as \eqref{weight-BD},
and due to Kurtz and Wheeden \cite{Kurtz-Wheeden}.

%%%==================================================================
\begin{table}[htbp]
\begin{tabular}{|c|c|c|}
\hline
{}
&$A_{p_1s/(Nn)} \times \dots \times A_{p_Ns/(Nn)}$
&$A_{(p_1s/(Nn),\dots,p_Ns/(Nn))}$
\\ \hline
$W^s(\R^{Nn})$
&hold
&hold
\\ \hline
$W^{(s/N,\dots,s/N)}(\R^{Nn})$
&hold
&?
\\
\hline
\end{tabular}
\label{Table-1}
\caption{Weighted estimates from
$L^{p_1}(w_1) \times \dots \times L^{p_N}(w_N)$
to $L^p(\nu_{\vec{w}})$.}
\end{table}
%%%==================================================================

The purpose of this paper is to answer the question
whether estimate \eqref{weight-FT} holds
under the condition
$\vec{w}=(w_1,\dots,w_N) \in A_{(p_1s/(Nn),\dots,p_Ns/(Nn))}$,
and the main result is the following:

%%%==================================================================
\begin{thm}\label{main}
Let $N \ge 2$, $Nn/2 <s \le Nn$, 
$Nn/s<p_1,\dots,p_N<\infty$
and $1/p_1+\dots+1/p_N=1/p$.
Then there exists
$\vec{w}=(w_1,\dots,w_N) \in A_{(p_1s/(Nn),\dots,p_Ns/(Nn))}$
such that the estimate
\[
\|T_m\|_{L^{p_1}(w_1) \times \dots \times L^{p_N}(w_N) \to L^p(\nu_{\vec{w}})}
\lesssim
\sup_{j \in \Z}\|m_j\|_{W^{(s/N,\dots,s/N)}(\R^{Nn})}
\]
does not hold, where the implicit constant is independent of $m$.
\end{thm}
%%%==================================================================

It should be pointed out that
the statement similar to Theorem \ref{main}
holds even if we replace
$L^p(\nu_{\vec{w}})$ by $L^{p,\infty}(\nu_{\vec{w}})$
(see Remark \ref{remark-weak}).
Using the class $A_{\vec{P}/\vec{Q}}$
which coincides with $A_{(p_1s/(Nn),\dots,p_Ns/(Nn))}$
if $\vec{P}=(p_1,\dots,p_N)$ and $\vec{Q}=(Nn/s,\dots,Nn/s)$,
Jiao \cite{Jiao} gave a generalization of \eqref{weight-BD}.
See Remark \ref{remark-generalization} for the result
corresponding to this weight class.

%%%==================================================================
%%%==================================================================
\section{Preliminaries}\label{section2}
%%%==================================================================
For two non-negative quantities $A$ and $B$,
the notation $A \lesssim B$ means that
$A \le CB$ for some unspecified constant $C>0$,
and $A \approx B$ means that
$A \lesssim B$ and $B \lesssim A$.
For $1<p<\infty$,
$p'$ is the conjugate exponent of $p$, 
that is, $1/p+1/p'=1$.

Let $\calS(\R^n)$ be the Schwartz space of
all rapidly decreasing smooth functions.
We define the Fourier transform $\calF f$
and the inverse Fourier transform $\calF^{-1}f$
of $f \in \calS(\R^n)$ by
\[
\calF f(\xi)
=\widehat{f}(\xi)
=\int_{\R^n}e^{-ix\cdot\xi} f(x)\, dx
\quad \text{and} \quad
\calF^{-1}f(x)
=\frac{1}{(2\pi)^n}
\int_{\R^n}e^{ix\cdot \xi} f(\xi)\, d\xi.
\]

The (usual) Sobolev space $W^s(\R^{Nn})$,
$s \in \R$,
is defined by the norm
\[
\|F\|_{W^s(\R^{Nn})}
=\left(\int_{\R^{Nn}}
(1+|\xi|^2)^s|\widehat{F}(\xi)|^2\, d\xi \right)^{1/2},
\]
where $\widehat{F}$ is the Fourier transform
in all the variables.
The Sobolev space of product type $W^{(s_1,\dots,s_N)}(\R^{Nn})$,
$s_1,\dots,s_N \in \R$,
is also defined by the norm
\[
\|F\|_{W^{(s_1,\dots,s_N)}(\R^{Nn})}
=\left(\int_{\R^{Nn}}
(1+|\xi_1|^2)^{s_1} \dots (1+|\xi_N|^2)^{s_N}
|\widehat{F}(\xi)|^2\, d\xi \right)^{1/2},
\]
where $\xi=(\xi_1,\dots,\xi_N) \in \R^n \times \dots \times \R^n$.

Let $w \ge 0$.
For a measurable set $E$,
we write $w(E)=\int_E w(x)\, dx$,
and simply $|E|=\int_E dx$ for the case $w \equiv 1$.
The weighted Lebesgue space $L^p(w)$,
$0<p<\infty$,
consists of all measurable functions $f$ on $\R^n$
such that
\[
\|f\|_{L^p(w)}
=\left(\int_{\R^n}|f(x)|^p w(x)\, dx \right)^{1/p}<\infty.
\]
The weighted weak Lebesgue space $L^{p,\infty}(w)$
is also defined by the norm
\[
\|f\|_{L^{p,\infty}(w)}
=\sup_{\lambda>0}
\lambda
\left\{w\left(\left\{x \in \R^n \,:\,
|f(x)| > \lambda \right\}\right)\right\}^{1/p}.
\]

We say that a weight $w$ belongs to
the Muckenhoupt class $A_{p}$, $1<p<\infty$, if
\[
\sup_{Q}\left( \frac{1}{|Q|}\int_{Q}w(x) \, dx\right)
\left( \frac{1}{|Q|}\int_{Q}w(x)^{1-p'} \, dx\right)^{p-1} < \infty,
\]
where the supremum is taken over all cubes $Q \subset \R^n$
with sides parallel to the axes.
We also say that $\vec{w}=(w_1,\dots,w_N)$ belongs to
the class $A_{(p_1,\dots,p_N)}$,
$1<p_1,\dots,p_N<\infty$, if
\[
\sup_Q
\left(\frac{1}{|Q|}\int_Q \nu_{\vec{w}}(x)\, dx\right)^{1/p}
\prod_{k=1}^N
\left(\frac{1}{|Q|}\int_Q w_k(x)^{1-p_k'}\, dx\right)^{1/p_k'}
< \infty,
\]
where $1/p_1+\dots+1/p_N=1/p$ and
$\nu_{\vec{w}}=\prod_{k=1}^N w_k^{p/p_k}$.

The proof of the following lemma is based on the argument
of \cite[Example 9.1.7]{Grafakos}.
%%%==================================================================
\begin{lem}\label{range}
Let $N \ge 2$, $1<p_1,\dots,p_N<\infty$ and $1/p_1+\dots+1/p_N=1/p$.
If $\alpha_1,\alpha_2$ satisfy
$\alpha_1/p_1+\alpha_2/p_2>-n/p$
and $\alpha_k<n(p_k-1)$ for $k=1,2$,
then
\[
\vec{w}=(w_1,w_2,w_3,w_4,\dots,w_N)
=(|x|^{\alpha_1},|x|^{\alpha_2},1,1,\dots,1)
\]
belongs to the class $A_{(p_1,\dots,p_N)}$.
\end{lem}
%%%==================================================================
\begin{proof}
Since $w_k=1$ for $k \ge 3$,
the desired conclusion follows from
\[
\sup_B
\left(\frac{1}{|B|}\int_B w_1^{p/p_1}w_2^{p/p_2}\, dx\right)^{1/p}
\prod_{k=1}^2
\left(\frac{1}{|B|}\int_B w_k^{1-p_k'}\, dx\right)^{1/p_k'}
< \infty,
\]
where the supremum is taken over all balls $B$ in $\R^n$
(instead of cubes).
Let $B$ be the ball with center $x_0$ and radius $r$.

We first consider the case $|x_0| \ge 2r$.
In this case, $|x| \approx |x_0|$ for all $x \in B$.
Then
\begin{align*}
&\left(\frac{1}{|B|}\int_B |x|^{p(\alpha_1/p_1+\alpha_2/p_2)}\,
dx\right)^{1/p}
\prod_{k=1}^2
\left(\frac{1}{|B|}\int_B |x|^{\alpha_k(1-p_k')}\, dx\right)^{1/p_k'}
\\
&\approx |x_0|^{\alpha_1/p_1+\alpha_2/p_2}
|x_0|^{\alpha_1(1/p_1'-1)}|x_0|^{\alpha_2(1/p_2'-1)}
=1.
\end{align*}

We next consider the case $|x_0|<2r$.
In this case, $B \subset \{x \in \R^n \,:\, |x|<3r\}$.
Since $p(\alpha_1/p_1+\alpha_2/p_2)>-n$
and $\alpha_k(1-p_k')>-n$ for $k=1,2$,
we have
\[
\int_B |x|^{p(\alpha_1/p_1+\alpha_2/p_2)}\, dx
\le \int_{|x| <3r} |x|^{p(\alpha_1/p_1+\alpha_2/p_2)}\, dx
\lesssim r^{p(\alpha_1/p_1+\alpha_2/p_2)+n}
\]
and
\[
\int_B |x|^{\alpha_k(1-p_k')}\, dx
\le \int_{|x|<3r} |x|^{\alpha_k(1-p_k')}\, dx
\lesssim r^{\alpha_k(1-p_k')+n},
\qquad k=1,2.
\]
Hence,
\begin{align*}
&\left(\frac{1}{|B|}\int_B |x|^{p(\alpha_1/p_1+\alpha_2/p_2)}\,
dx\right)^{1/p}
\prod_{k=1}^2
\left(\frac{1}{|B|}\int_B |x|^{\alpha_k(1-p_k')}\, dx\right)^{1/p_k'}
\\
&\lesssim r^{\alpha_1/p_1+\alpha_2/p_2}
r^{\alpha_1(1/p_1'-1)}r^{\alpha_2(1/p_2'-1)}
=1.
\end{align*}
The proof is complete.
\end{proof}
%%%==================================================================

The following fact is known,
but we shall give a proof for the reader's convenience.
%%%==================================================================
\begin{lem}\label{existence}
Let $r>0$, and let $\ell$ be a non-negative integer.
Then there is a function $\varphi \in \calS(\R^n)$ so that
$\mathrm{supp}\, \varphi \subset \{x \in \R^n \,:\, |x| \le r\}$,
$\int_{\R^n}\varphi(x)^2\, dx \neq 0$
and $\int_{\R^n}x^{\beta}\varphi(x)\, dx=0$
for all multi-indices $\beta$ satisfying $|\beta| \le \ell$.
\end{lem}
%%%==================================================================
\begin{proof}
Let $\psi \in \calS(\R^n) \setminus \{0\}$
be a real valued function satisfying
$\mathrm{supp}\, \psi \subset \{x \in \R^n \,:\, |x| \le r\}$,
and set $\varphi(x)=(-\Delta)^{\ell+1}\psi(x)$,
where
$\Delta=\partial^2/\partial x_1^2+\dots+\partial^2/\partial x_n^2$.

We shall check that $\varphi$ satisfies all the required conditions.
Obviously, $\mathrm{supp}\, \varphi \subset \{x \in \R^n \,:\, |x| \le r\}$.
Since $\psi$ is not identically equal to zero,
so is $\widehat{\psi}$.
Thus, we can take $\xi_0 \in \R^n$ and $r_0>0$
such that $\widehat{\psi}(\xi) \neq 0$
if $|\xi-\xi_0| \le r_0$.
Since $\varphi$ is a real valued function,
we have by Plancherel's theorem
\begin{align*}
\int_{\R^n}\varphi(x)^2\, dx
&=\int_{\R^n}\left|
(-\Delta)^{\ell+1}\psi(x)\right|^2\, dx
=\frac{1}{(2\pi)^n}
\int_{\R^n}\left|
|\xi|^{2(\ell+1)}\widehat{\psi}(\xi)
\right|^2 \, d\xi
\\
&\ge \frac{1}{(2\pi)^n}
\int_{|\xi-\xi_0| \le r_0}\left|
|\xi|^{2(\ell+1)}\widehat{\psi}(\xi)
\right|^2 \, d\xi
\gtrsim \int_{|\xi-\xi_0| \le r_0}
|\xi|^{4(\ell+1)}
\, d\xi
\neq 0.
\end{align*}
Finally,
\begin{align*}
\int_{\R^n}(-ix)^{\beta}\varphi(x)\, dx
=\partial^{\beta}\widehat{\varphi}(\xi)
\Big|_{\xi=0}
=\partial^{\beta}
\left(|\xi|^{2(\ell+1)}\widehat{\psi}(\xi)\right)
\Big|_{\xi=0}=0
\end{align*}
for $|\beta| \le \ell$.
This completes the proof.
\end{proof}
%%%==================================================================

%%%==================================================================
%%%==================================================================
\section{Proof of Theorem \ref{main}}\label{section3}
%%%==================================================================
In this section,
using the ideas given in
\cite[Section 7]{Grafakos-Miyachi-Tomita, Miyachi-Tomita}
and \cite[Remark 7.5]{LOPTT},
we shall prove Theorem \ref{main}.

%%%==================================================================
\begin{proof}[Proof of Theorem \ref{main}]
Let $N \ge 2$, $Nn/2 <s \le Nn$, 
$Nn/s<p_1,\dots,p_N<\infty$
and $1/p_1+\dots+1/p_N=1/p$.
We first claim that there exist
$\alpha_1<-n$ and $\alpha_2>-n$ such that
\begin{equation}\label{alpha_12}
\alpha_1/p_1+\alpha_2/p_2>-n/p,
\qquad
\alpha_k/p_k<s/N-n/p_k,
\quad k=1,2,
\end{equation}
and
\begin{equation}\label{alpha_1}
\alpha_1/p_1<-n/p_1-s/N+n/2.
\end{equation}
Indeed,
since $-n/p+n/p_1+s/N-n/2<s/N-n/p_2$
and $s/N-n/p_2>0$,
we can take $\alpha_2 \ge 0$ satisfying
$-n/p+n/p_1+s/N-n/2<\alpha_2/p_2<s/N-n/p_2$.
Then, since $-\alpha_2/p_2-n/p<-n/p_1-s/N+n/2$,
we can take $\alpha_1$ satisfying
$-\alpha_2/p_2-n/p<\alpha_1/p_1<-n/p_1-s/N+n/2$.
It is easy to check that these $\alpha_1,\alpha_2$
satisfy
$\alpha_1<-n$, $\alpha_2>-n$,
\eqref{alpha_12} and \eqref{alpha_1}.

For $\alpha_1<-n$ and $\alpha_2>-n$
satisfying \eqref{alpha_12} and \eqref{alpha_1},
we set
\begin{equation}\label{power-weight}
\vec{w}=(w_1,w_2,w_3,w_4,\dots,w_N)
=(|x|^{\alpha_1},|x|^{\alpha_2},1,1,\dots,1).
\end{equation}
Let $(q_1,\dots,q_N)=(p_1s/(Nn),\dots,p_Ns/(Nn))$
and $1/q_1+\dots+1/q_N=1/q$.
Since $p/p_k=q/q_k$ for $k=1,2$,
it follows from \eqref{alpha_12} that
$\alpha_1/q_1+\alpha_2/q_2>-n/q$
and $\alpha_k<n(q_k-1)$ for $k=1,2$.
Then, by Lemma \ref{range},
we see that $\vec{w} \in A_{(q_1,\dots,q_N)}$.

We shall prove Theorem \ref{main} with
$\vec{w}$ defined by \eqref{alpha_12},
\eqref{alpha_1} and \eqref{power-weight}
by contradiction.
To do this, we assume that the estimate
\begin{equation}\label{assumption}
\|T_m\|_{L^{p_1}(w_1) \times \dots \times L^{p_N}(w_N) \to L^p(\nu_{\vec{w}})}
\lesssim
\sup_{j \in \Z}\|m_j\|_{W^{(s/N,\dots,s/N)}}
\end{equation}
holds, where the implicit constant is independent of $m$.
Let $\widehat{\varphi} \in \calS(\R^n)$ be a function
as in Lemma \ref{existence} with
$r=1/(10N)$ and $\ell$ satisfying $p_1(\ell+1)+\alpha_1>-n$:
$\mathrm{supp}\, \widehat{\varphi} \subset
\{\eta \in \R^n \,:\, |\eta| \le 1/(10N)\}$,
\begin{align}
&\int_{\R^n}\widehat{\varphi}(\eta)^2\, d\eta \neq 0,
\label{not-vanish} \\
&\int_{\R^n}\eta^{\beta}\widehat{\varphi}(\eta)\, d\eta=0,
\quad |\beta| \le \ell.
\label{vanish}
\end{align}
For sufficiently small $\epsilon>0$,
we set
\begin{equation}\label{example-m}
m^{(\epsilon)}(\xi)
=\widehat{\varphi}((\xi_1-e_1)/\epsilon)
\widehat{\varphi}(\xi_2)
\widehat{\varphi}(\xi_3) \times \dots \times
\widehat{\varphi}(\xi_N),
\end{equation}
where $e_1=(1,0,0,\dots,0) \in \R^n$.

We shall estimate the Sobolev norm of
$m^{(\epsilon)}$ as follows:
\begin{equation}\label{estimate-m}
\sup_{j \in \Z}\|m^{(\epsilon)}_j\|_{W^{(s/N,\dots,s/N)}}
\lesssim \epsilon^{-s/N+n/2},
\end{equation}
where $m^{(\epsilon)}_j$ is defined by
\eqref{(multiplier)} with $m$ replaced by $m^{(\epsilon)}$.
To do this, we choose the function $\Psi \in \calS (\R^{Nn})$
appearing in the definition of $m^{(\epsilon)}_j$ so that
\begin{align*}
&\mathrm{supp}\, \Psi \subset \{\xi\in \R^{Nn}
\,:\,
2^{-1/2-\gamma} \le |\xi| \le 2^{1/2+\gamma}\}, 
\\
&\Psi (\xi) =1
\quad \text{if} \quad
2^{-1/2 +\gamma}\le |\xi| \le 2^{1/2-\gamma},
\end{align*}
where $\gamma>0$ is a sufficiently small number. 
If $\epsilon>0$ is sufficiently small, then
\begin{align*}
\mathrm{supp}\, m^{(\epsilon)}
&\subset 
\{(\xi_1,\dots,\xi_N)
\,:\,
|\xi_1-e_1| \le \epsilon/(10N), \; 
|\xi_k| \le 1/(10N), \ 2 \le k \le N\}
\\
&
\subset 
\{(\xi_1,\dots,\xi_N)
\,:\,  2^{-1/2 +\gamma}\le
(|\xi_1|^2+\dots+|\xi_N|^2)^{1/2} \le 2^{1/2-\gamma}\} .
\end{align*}
This implies
\[
m^{(\epsilon)}_j(\xi)
=m^{(\epsilon)}(2^j \xi )\Psi (\xi)=
\begin{cases}
m^{(\epsilon)}(\xi) & \text{if} \quad j=0
\\
0 & \text{if} \quad j\neq 0,
\end{cases}
\]
and consequently
\begin{align*}
\sup_{j\in \Z}
\|m^{(\epsilon)}_j\|_{W^{(s/N,\dots,s/N)}}
=\|m^{(\epsilon)}\|_{W^{(s/N,\dots,s/N)}}
=\|\widehat{\varphi}((\cdot-e_1)/\epsilon)\|_{W^{s/N}}
\|\widehat{\varphi}\|_{W^{s/N}}^{N-1} .
\end{align*}
Taking a sufficiently large $L>0$, we have
\begin{align*}
&\|\widehat{\varphi} ((\cdot-e_1) /\epsilon)\|_{W^{s/N}}
=(2\pi)^{n}\|(1+|\cdot|^2)^{s/(2N)}\epsilon^n
\varphi(\epsilon \, \cdot)\|_{L^2}
\\
&\lesssim \epsilon^n 
\left(
\int_{\R^n} (1+|x|)^{2s/N} (1+\epsilon |x|)^{-2L}\,dx
\right)^{1/2}
\\
&\lesssim \epsilon^n 
\left(\int_{|x|\le 1} \,dx
+\int_{1<|x|\le 1/\epsilon}
|x|^{2s/N}\,dx 
+\int_{|x|>1/\epsilon}
|x|^{2s/N}(\epsilon |x|)^{-2L}\,dx 
\right)^{1/2}
\\
&\lesssim \epsilon^{-s/N +n/2},
\end{align*}
and we obtain \eqref{estimate-m}.

Let $\psi \in \calS(\R^n)$ be such that
$\widehat{\psi}=1$ on $\mathrm{supp}\, \widehat{\varphi}$,
and set
\begin{equation}\label{example-f}
\widehat{f_1}(\xi_1)=
\widehat{f_1^{(\epsilon)}}(\xi_1)
=\epsilon^{n/p_1-n}\widehat{\varphi}((\xi_1-e_1)/\epsilon),
\qquad
\widehat{f_k}(\xi_k)=\widehat{\psi}(\xi_k),
\ 2 \le k \le N.
\end{equation}
To estimate $\|f_{1}^{(\epsilon)}\|_{L^{p_1}(w_1)}$,
we check that $\varphi$ belongs to $L^{p_1}(w_1)$.
It follows from \eqref{vanish} that
$\partial^{\beta}\varphi(0)=0$
for $|\beta| \le \ell$.
Combining this with Taylor's formula,
we see that $|\varphi(x)| \lesssim |x|^{\ell+1}$.
Then,
since $p_1(\ell+1)+\alpha_1>-n$,
\begin{align*}
\int_{|x|<1}
\left|\varphi(x) \right|^{p_1}|x|^{\alpha_1} dx
\lesssim \int_{|x|<1}
|x|^{p_1(\ell+1)+\alpha_1}\, dx<\infty.
\end{align*}
On the other hand, it is obvious that
\[
\int_{|x| \ge 1}
\left|\varphi(x) \right|^{p_1}|x|^{\alpha_1} dx
<\infty.
\]
Consequently, $\varphi$ belongs to $L^{p_1}(w_1)$.
Hence,
\begin{equation}\label{estimate-f}
\|f_1^{(\epsilon)}\|_{L^{p_1}(w_1)}
=\left(\int_{\R^n}
|\epsilon^{n/p_1}\varphi(\epsilon x)|^{p_1}
|x|^{\alpha_1}\, dx\right)^{1/p_1}
=\epsilon^{-\alpha_1/p_1}\|\varphi\|_{L^{p_1}(w_1)}.
\end{equation}
The condition $\alpha_2>-n$ implies that
$w_2=|x|^{\alpha_2}$ is locally integrable.
Then,
since $\psi$ is rapidly decreasing,
we have
$\|f_2\|_{L^{p_2}(w_2)}
=\|\psi\|_{L^{p_2}(|x|^{\alpha_2})}<\infty$
and
$\|f_k\|_{L^{p_k}(w_k)}
=\|\psi\|_{L^{p_k}}<\infty$
for $k=3,\dots,N$.

We shall finish the proof.
By \eqref{example-m} and \eqref{example-f},
\begin{align}\label{T_m-form}
T_{m^{(\epsilon)}}(\vec{f})(x)
&=\calF^{-1}[\widehat{\varphi}((\cdot-e_1)/\epsilon)\widehat{f_1}](x)
\calF^{-1}[\widehat{\varphi}\widehat{f_2}](x)
\dots \calF^{-1}[\widehat{\varphi}\widehat{f_N}](x)
\\
&=\calF^{-1}[\epsilon^{n/p_1-n}
\widehat{\varphi}((\cdot-e_1)/\epsilon)^2](x)
\calF^{-1}[\widehat{\varphi}](x)
\dots \calF^{-1}[\widehat{\varphi}](x)
\nonumber \\
&=\epsilon^{n/p_1}e^{ie_1\cdot x}(\varphi*\varphi)(\epsilon x)
\varphi(x)^{N-1},
\nonumber
\end{align}
where $\calF^{-1}$ is the inverse Fourier transform on $\R^n$.
Since $\varphi$ is not identically equal to zero
and $\nu_{\vec{w}}=|x|^{p(\alpha_1/p_1+\alpha_2/p_2)}$ is locally integrable
(see \eqref{alpha_12}),
we can take $R>0$ satisfying
\begin{equation}\label{not-vanish-int}
0<\int_{|x| \le R}|\varphi(x)|^{p(N-1)}
\nu_{\vec{w}}(x)\, dx<\infty.
\end{equation}
On the other hand,
it follows from \eqref{not-vanish} that
\[
\varphi*\varphi(0)
=\calF^{-1}[\widehat{\varphi}\, \widehat{\varphi}](0)
=\frac{1}{(2\pi)^n}
\int_{\R^n}\widehat{\varphi}(\eta)^2\, d\eta
\neq 0.
\]
By the continuity of $\varphi*\varphi$ at the origin,
there exist $C>0$ and $\epsilon_0$ such that
\begin{equation}\label{estimate-phi-below}
|\varphi*\varphi(\epsilon x)|
\ge C
\qquad
\text{for all $0<\epsilon<\epsilon_0$ and $|x| \le R$}.
\end{equation}
Thus,
\begin{equation}\label{estimate-below-strong}
\|T_{m^{(\epsilon)}}(\vec{f})\|_{L^p(\nu_{\vec{w}})}
\ge C\epsilon^{n/p_1}
\left(\int_{|x| \le R}|\varphi(x)|^{p(N-1)}
\nu_{\vec{w}}(x)\, dx\right)^{1/p}
=C\epsilon^{n/p_1}
\end{equation}
for all $0<\epsilon<\epsilon_0$.
Hence,
by \eqref{assumption}, \eqref{estimate-m} and \eqref{estimate-f},
\begin{align*}
\epsilon^{n/p_1}
&\lesssim
\|T_{m^{(\epsilon)}}(\vec{f})\|_{L^p(\nu_{\vec{w}})}
\le
\|T_{m^{(\epsilon)}}\|_{L^{p_1}(w_1) \times \dots \times L^{p_N}(w_N) \to L^p(\nu_{\vec{w}})}
\prod_{k=1}^N\|f_k\|_{L^{p_k}(w_k)}
\\
&\lesssim
\sup_{j \in \Z}\|m_j^{(\epsilon)}\|_{W^{(s/N,\dots,s/N)}}
\|f_1^{(\epsilon)}\|_{L^{p_1}(w_1)}
\prod_{k=2}^N\|f_k\|_{L^{p_k}(w_k)}
\lesssim
\epsilon^{-s/N+n/2-\alpha_1/p_1}
\end{align*}
for all sufficiently small $\epsilon>0$.
However, since $n/p_1<-s/N+n/2-\alpha_1/p_1$
(see \eqref{alpha_1}),
this is a contradiction.
Therefore,
estimate \eqref{assumption} does not hold.
\end{proof}
%%%==================================================================

We end this paper by giving the two remarks
mentioned in the end of the introduction.

%%%==================================================================
\begin{rem}\label{remark-weak}
Let $N$, $s$ and $p_1,\dots,p_N$ satisfy
the assumption of Theorem \ref{main}.
Once inequality \eqref{estimate-below-strong}
is replaced by the sharper one
\begin{equation}\label{estimate-below-weak}
\|T_{m^{(\epsilon)}}(\vec{f})\|_{L^{p,\infty}(\nu_{\vec{w}})}
\gtrsim \epsilon^{n/p_1}
\qquad
\text{for all sufficiently small $\epsilon>0$},
\end{equation}
where $m^{(\epsilon)}$, $\vec{f}$ and $\vec{w}$ 
are the same as in the proof of Theorem \ref{main},
the same argument as before shows that the estimate
\[
\|T_m\|_{L^{p_1}(w_1) \times \dots \times L^{p_N}(w_N)
\to L^{p,\infty}(\nu_{\vec{w}})}
\lesssim
\sup_{j \in \Z}\|m_j\|_{W^{(s/N,\dots,s/N)}(\R^{Nn})}
\]
does not hold,
where the implicit constant is independent of $m$.

It is not difficult to prove \eqref{estimate-below-weak}.
Indeed, by \eqref{T_m-form} and \eqref{estimate-phi-below},
\[
\left\{x \in \R^n \,:\,
|T_{m^{(\epsilon)}}(\vec{f})(x)|
> \lambda \right\}
\supset
\left\{x \in B_R \,:\,
|\varphi(x)|^{N-1}
> (C\epsilon^{n/p_1})^{-1}\lambda \right\}
\]
for all $0<\epsilon<\epsilon_0$ and $\lambda>0$,
where $B_R$ is the ball with center at the origin and radius $R$.
Hence, since
$0<\sup_{\lambda>0}
\lambda
\left\{\nu_{\vec{w}}\left(\left\{x \in B_R \,:\,
|\varphi(x)|^{N-1} > \lambda \right\}\right)\right\}^{1/p}
<\infty$
(see \eqref{not-vanish-int}),
we obtain \eqref{estimate-below-weak}.
\end{rem}
%%%==================================================================

%%%==================================================================
\begin{rem}\label{remark-generalization}
Let $1 \leq q_k < p_k$, 
$k=1, \dots , N$, 
and set 
$\vec{P}=(p_1, \dots , p_N)$, 
$\vec{Q}=(q_1, \dots , q_N)$. 
We say that $\vec{w}=(w_1, \dots ,w_N)$
belongs to the class $A_{\vec{P}/\vec{Q}}$ if
\[
\sup_{Q}
\left( \frac{1}{|Q|} \int_{Q} \nu_{\vec{w}}(x) \,dx \right)^{1/p}
\prod_{k=1}^{N}
\left(
\frac{1}{|Q|}
\int_{Q} w_k^{1-(p_k/q_k)'} (x) \, dx
\right)^{1/q_k -1/p_k}
< \infty, 
\]
where $1/p_1+\dots+1/p_N=1/p$ and $\nu_{\vec{w}}=\prod_{k=1}^{N}w_k^{p/p_k}$.
Note that $\prod_{k=1}^{N} A_{p_k/q_k} \subset A_{\vec{P}/\vec{Q}}$,
and $A_{\vec{P}/\vec{Q}}=A_{(p_1/q_0, \dots, p_N/q_0)}$
if $q_1 = \dots = q_N = q_0 \geq 1$.

Let $N \geq 2$, 
$n/2 <s_k \leq n$,  
$n/s_k < p_k < \infty$ , 
$k=1, \dots, N$, 
$1/p_1 + \dots +1/p_N =1/p$,
and set $q_k=n/s_k$.
Jiao \cite{Jiao} proved that 
if $\vec{w} \in A_{\vec{P}/\vec{Q}}$,
then
\begin{equation*}
\|T_m \|
_{L^{p_1}(w_1) \times \dots \times L^{p_N}(w_N) \to L^p(\nu_{\vec{w}})}
\lesssim
\sup_{j \in \Z}
\| m_j \|
_{W^{s_1+ \dots +s_N} (\R^{Nn})}.
\end{equation*}
In the case $q_1= \dots =q_N =Nn/s$, 
this coincides with the results of \cite{Bui-Duong, Li-Sun}
(see \eqref{weight-BD}). 
However, 
in the same way as in the proof of Theorem \ref{main}, 
we can prove that the estimate 
\begin{equation*}
\|T_m \|
_{L^{p_1}(w_1) \times \dots \times L^{p_N}(w_N) \to L^p(\nu_{\vec{w}}) }
\lesssim
\sup_{j \in \Z}
\| m_j \|
_{W^{(s_1, \dots ,s_N)} (\R^{Nn})}
\end{equation*}
does not in general hold
for $\vec{w} \in A_{\vec{P}/\vec{Q}}$
(but this estimate holds 
for $\vec{w} \in \prod_{k=1}^{N} A_{p_k/q_k}$,
see \eqref{weight-FT}
and also \cite[Therem 6.2]{Fujita-Tomita} for general case). 
Indeed, as for \eqref{alpha_12} and \eqref{alpha_1}, 
we can choose
$\alpha_1$ and $\alpha_2$ so that 
$\alpha_1/p_1 +\alpha_2/p_2 >-n/p$, 
$\alpha_k /p_k < s_k -n/p_k$, 
$k=1, 2$, 
and
$\alpha_1/p_1<-n/p_1-s_1+n/2$. 
Then
$(|x|^{\alpha_1}, |x|^{\alpha_2}, 1, \dots, 1) \in A_{\vec{P}/\vec{Q}}$,
and the rest of the proof is similar. 
\end{rem}
%%%==================================================================

%%%==================================================================
%%%==================================================================

%%%==================================================================
%%%==================================================================
\end{document}